\documentclass[review]{elsarticle}

\usepackage{lineno,hyperref}
\modulolinenumbers[5]

\journal{Journal of \LaTeX\ Templates}

\usepackage{amssymb,amsmath,amsthm, latexsym}

\theoremstyle{plain}
\newtheorem{thm}{{\sc Theorem}}[section]
\newtheorem{lem}[thm]{\sc Lemma}

\newtheorem{cor}[thm]{\sc Corollary}
\newtheorem{obs}[thm]{\sc Observation}
\newtheorem{ex}[thm]{\sc Example}
\newtheorem{defn}[thm]{\sc Definition}

\newcounter{num}
\renewcommand{\bar}{\overline}

\usepackage{amsmath,mathdots}

\begin{document}

\begin{frontmatter}

\title{Pascal Eigenspaces and Invariant Sequences of the First or Second Kind\tnoteref{mytitlenote}}
\tnotetext[mytitlenote]{Research supported by Daegu University Research Grant 2013}


\author[daegu]{Ik-Pyo Kim\corref{mycorrespondingauthor}}
\ead{kimikpyo@daegu.ac.kr}
\cortext[mycorrespondingauthor]{Corresponding author}
\address[daegu]{Department of Mathematics Education,
     Daegu University, Gyeongbuk, 38453, Republic of Korea}


\author[washington]{Michael J. Tsatsomeros}
\address[washington]{ Department of Mathematics and Statistics,
     Washington State University, Pullman, WA 99164, USA}
\ead{tsat@math.wsu.edu}

\begin{abstract}
An infinite real sequence $\{a_n\}$ is called an invariant sequence of the first
(resp., second) kind if $a_n=\sum_{k=0}^n {n \choose k} (-1)^k a_k$
(resp., $a_n=\sum_{k=n}^{\infty} {k \choose n} (-1)^k a_k$). We review and investigate
invariant sequences of the first and second kinds, and study their relationships
using similarities of Pascal-type matrices and their eigenspaces.
\end{abstract}

\end{frontmatter}

\linenumbers

\section{Introduction}
\par
\setcounter{num}{1}
\setcounter{equation}{0}
Inverse relations play an important role in combinatorics \cite{Riordan}.
The {\em binomial inversion formula}, which states that for sequences $\{a_n\}$ and $\{b_n\}$ ($n=0,1, 2, \ldots$),
\begin{equation}
\label{b1}
a_n=\sum_{k=0}^{n} {n \choose k}(-1)^k b_k
\;\;\;\mbox{\rm if and only if}\;\;\;
b_n=\sum_{k=0}^{n} {n \choose k}(-1)^k a_k,
\end{equation}
is a typical inverse relation of interest in \cite{Choi,Donaghey,Prodinger,Sun,Wang,Wojcik}. Specifically,
(\ref{b1}) motivated Sun \cite{Sun} to investigate the following sequences.
\begin{defn}
{\rm
Let $\{a_n\}$ ($n=0,1, 2, \ldots$) be a sequence such that
\begin{equation}
\label{first}
(-1)^{s-1} \, a_n=\sum_{k=0}^{n} {n \choose k}(-1)^k a_k; \;\;s=1 \;\;\mbox{or}\;\; s=2.
\end{equation}
We refer to $\{a_n\}$ as an {\em invariant sequence} (when $s=1$) or an
{\em inverse invariant sequence} (when $s=2$) {\em of the first kind}.
}
\end{defn}
Several examples of invariant sequences of the first kind can be found in \cite{Sun}, including
$$ \{\frac{1}{2^n}\}, \; \{ nF_{n-1}\},
\; \{L_n\},\; \{(-1)^nB_n\} \;\; (n\geq 0),
$$
where $F_{-1}=0$ and $\{F_n\}$, $\{L_n\}$, and $\{B_n\}$ are the Fibonacci sequence, Lucas sequence,
and Bernoulli numbers \cite{Comtet}, respectively.
In this paper, we will establish (see Lemma \ref{bf2proof}) the {\em modified binomial inversion formula} such that
\begin{equation}
\label{b2}
a_n=\sum_{k=n}^{\infty} {k \choose n}(-1)^k b_k
       \;\;\;\mbox{\rm if and only if}\;\;\;
   b_n=\sum_{k=n}^{\infty} {k \choose n}(-1)^k a_k.
\end{equation}
Motivated by (\ref{b2}), we will introduce and consider the following sequences.
\begin{defn}
{\rm
Let $\{a_n\}$ ($n=0, 1, 2, \ldots$) be a sequence such that
\begin{equation}
\label{second}
(-1)^{s-1} a_n = \sum_{k=n}^{\infty} {k \choose n}(-1)^k a_k; \;\; s=1 \;\;\mbox{or}\;\; s=2.
\end{equation}
We refer to $\{a_n\}$ as an {\em invariant sequence} (when $s=1$) or an {\em inverse invariant sequence}
(when $s=2$) {\em of the second kind}.
}
\end{defn}
Naturally arising are the questions of existence, identification, and construction of (inverse) invariant sequences
of the second kind, as well as the problem of characterizing such sequences and examining their
relationship to their counterparts of the first kind.
Invariant sequences, which are also called self-inverse sequences in \cite{Wang}, have indeed been studied by several authors
\cite {Choi,Donaghey,Sun,Wang}. They are naturally connected to involutory (also known as involution or self-invertible) matrices
\cite{Horn} and to Riordan involutions \cite{Cheon}. Involutory matrices find use in numerical methods for
differential equations \cite{Andrew, Horn}.
They are also useful in cryptography, information theory, and computer security by providing convenient encryption and decryption methods \cite{Acharya}.
Motivated by Shapiro's open questions \cite{Shapiroq}, Riordan involutions have been intensely investigated as a combinatorial
concept \cite{Cameron,Cheon}.
In this paper, we investigate invariant sequences by means of the eigenspaces of $PD$ and $P^TD$, where $P$ is the Pascal matrix and $D$
an infinite diagonal matrix with alternating diagonal entries in $\{1, \,-1\}$ (see Sections 2, 3). In fact, $PD$ and $P^TD$ are
involutory matrices and $PD$ is a Riordan involution. 
Our investigation follows the  ideas and connections of invariant sequences to the eigenspaces of $PD$ and $P^TD$
developed in Choi et al. \cite{Choi}. This will allow us to associate (inverse) invariant sequences of the first
and second kinds, as well as identify and construct such sequences (Section 4).
\section{Notation and preliminaries}
\par
\setcounter{num}{2}
\setcounter{equation}{0}
The following notation and conventions are used throughout the manuscript.
\begin{itemize}
\item
The infinite matrices in this paper have infinite numbers of rows $i$ and columns
$j$, with $i,j\in\{0,1, 2, \ldots\}$.
\item
$\mathbf{E}_\lambda (A)$ denotes the eigenspace of a (finite or infinite)
matrix $A$ corresponding to its eigenvalue $\lambda$.
\item
For a matrix $A$ with columns $A_j$ ($j=0,1, 2, \ldots$) and with
${\mathbf 0}_j$ denoting the vector of zeros in $\mathbb{R}^j$,
${A\hspace{0.01cm}^\downarrow}$ denotes the matrix whose $j$th column
is
${\displaystyle
\left[\begin{matrix}
\mathbf 0_j\\
A_j
\end{matrix} \right]
}$ where $\mathbf 0_0$ is vacuous.
\item
For a matrix $A$, its (possibly infinite) row and column index sets are $J$ and $K$,
respectively. For $J_0 \subseteq J, K_0 \subseteq K$, let
$A(J_0|K_0)$ denote the matrix obtained from $A$ by deleting rows in $J_0$ and
columns in $K_0$, and let $A[J_0|K_0]$ denote the matrix
$A(\bar{J_0}|\bar{K_0})$, where $\bar{J_0}=J\setminus J_0, \bar{K_0}=K\setminus K_0$.
For brevity, write $A(\cdot|K_0)$ and $A(J_0|\cdot)$ in place of $A(\emptyset|K_0)$
and $A(J_0|\emptyset)$, respectively.
Further, for $m,n\in\{0,1, 2, \ldots\}$, we let $A_{m,n}=A[\{0,1, 2, \ldots, m\}|\{0,1, 2, \ldots, n\}]$;
$A_{m,m}$ is abbreviated by $A_m$.
\item
The binomial coefficient (``$i$ choose $j$") is denoted by $i\choose j$
with the convention that it equals $0$  when $i<j$ or $j<0$.
\item
$P=\left[\begin{matrix}
         {i}\choose{j}
         \end{matrix} \right]$
$(i,j=0, 1, 2, \ldots)$ denotes the (infinite) {\em Pascal matrix}.
\item
$D={\rm diag}(1, -1, 1, -1, \ldots)$.
\item
Infinite real sequences $\{x_n\}$ are identified with the infinite
dimensional real vector space $\mathbb{R}^\infty$ consisting
of column vectors $\mathbf{x}=[x_0, x_1, x_2, \ldots ]^T$.
\end{itemize}

Notice that as a consequence of the binomial inversion formula (\ref{b1}), we have that
$$ P^{-1}=DPD=
\left[ (-1)^{i-j}
\begin{matrix}
 {i}\choose{j}
\end{matrix} \right] \;\;
(i,j=0, 1, 2, \ldots).$$
Thus $(PD)^{-1}=PD$ and (\ref{b1}) can be converted \cite{Choi} into a vector equation for
$\mathbf{x}=[a_0,a_1, a_2, \ldots]^T$ and $\mathbf{y}=[b_0,b_1, b_2, \ldots]^T \in \mathbb{R}^\infty$,
as follows:
\begin{equation}
\label{vbf}
PD\mathbf{x}=\mathbf{y}
\;\;\; \mbox{\rm if and only if} \;\;\;
PD\mathbf{y}=\mathbf{x}.
\end{equation}
\begin{lem}
\label{bf2proof}
Let $P$ and $D$ be the Pascal matrix and the diagonal matrix defined above,
and let $\mathbf{x}, \mathbf{y} \in \mathbb{R}^\infty$. Then
\begin{equation}
\label{vbf2}
P^TD\mathbf{x}=\mathbf{y}
\;\;\;\mbox{\rm if and only if}\;\;\;
P^TD\mathbf{y}=\mathbf{x},
\end{equation}
and the modified binomial inversion formula {\rm (\ref{b2})} holds.
\end{lem}
\begin{proof}
As $P^{-1}=DPD$, we have $(P^TD)^{-1}=D(P^T)^{-1}=P^TD$. As a consequence, (\ref{vbf2}) holds.
Letting $\mathbf{x}=[a_0, a_1, a_2, \ldots]^T$ and $\mathbf{y}=[b_0, b_1, b_2, \ldots]^T$ implies (\ref{b2}).
\end{proof}
Let $\mathbf{F}=[F_0, F_1, F_2, \ldots ]^T$ and $\mathbf{L}=[L_0, L_1, L_2, \ldots]^T$
denote the vectors in $\mathbb{R}^{\infty}$ whose entries are the members of the Fibonacci and Lucas
sequences, respectively; that is
\[ F_0=0, \; F_1=1, \; F_n=F_{n-1}+F_{n-2} \;\; (n\geq 2), \]
\[ L_0=2, \; L_1=1, \; L_n=L_{n-1}+L_{n-2} \;\; (n\geq 2). \]
The generating functions of $\mathbf{F}$ and $\mathbf{L}$
are $h_1(x)={x \over {1-x-x^2}}$ and $h_2(x)={{2-x}\over {1-x-x^2}}$, respectively \cite{Brualdi, Comtet}.

The following fact is known, however we include a proof for completeness.
\begin{lem}
	\label{wellknown}
	$PD\,\mathbf{F} = - \mathbf{F} \;\;\mbox{and}\;\; PD\,\mathbf{L} = \mathbf{L}.$
\end{lem}
\begin{proof}
Let $g(x)={1 \over {1-x}}$ and $f(x)={-x \over {1-x}}$.
For $j=0, 1, 2, \ldots$, the generating function of the $j$th column of $PD$ is $g(x)f(x)^j$ \cite{Shapiro}.
Thus the generating function of $PD\mathbf F$ is
\begin{equation*}
\begin{split}
[g(x), g(x)f(x), g(x)f(x)^2, \ldots][F_0, F_1, F_2, \ldots]^T&=g(x)(F_0+F_1f(x)+F_2f(x)^2+\cdots)\\
                                                             &=g(x)h_1(f(x))=-h_1(x),
\end{split}
\end{equation*}
which implies that $PD\,\mathbf{F} = - \mathbf{F}$. The proof of $PD\,\mathbf{L} = \mathbf{L}$ is similar.
\end{proof}
That is, $-1$ and $1$ are eigenvalues of $PD$ and consequently of $P^TD$.
In fact, these are the only eigenvalues of $PD$ and $P^TD$; see \cite{Choi}.
The corresponding eigenspaces are infinite dimensional. Indeed, if we
consider the Pascal-type matrices ${P\hspace{0.01cm}^\downarrow}$ and
${Q\hspace{0.01cm}^\downarrow}$ constructed via the Pascal matrix $P$ and
the matrix
\[ Q=P+\left[
       \begin{array}{cc}
1 & {\mathbf 0}^T\\
{\mathbf 0}& P
   \end{array}\right],
\]
then, as shown in \cite{Choi}, the columns of
\[
\left[\begin{matrix}
{\mathbf 0}^T\\
{P\hspace{0.01cm}^\downarrow}
\end{matrix} \right]\quad
\mbox{and} \quad
{Q\hspace{0.01cm}^\downarrow}
\]
form bases for $\mathbf{E}_{-1}(PD)$ and $\mathbf{E}_{1}(PD)$, respectively.
The following observation follows directly from the definitions and properties mentioned above.
\begin{obs}
\label{obs1}
The entries of $\mathbf{x} \in \mathbb{R}^\infty$ form
\begin{itemize}
\item
an invariant sequence of the first kind if and only if
$\mathbf{x} \in \mathbf{E}_{1} (PD)$;
\item
an inverse invariant sequence of the first kind if and only if
$\mathbf{x} \in \mathbf{E}_{-1} (PD)$;
\item
an invariant sequence of the second kind if and only if
$\mathbf{x} \in \mathbf{E}_{1} (P^TD)$;
\item
an inverse invariant sequence of the second kind if and only if
$\mathbf{x} \in \mathbf{E}_{-1} (P^TD)$.
\end{itemize}
\end{obs}
Based on Observation \ref{obs1}, our goal is to study the
eigenspaces $\mathbf{E}_\lambda (PD)$ and $\mathbf{E}_\lambda (P^TD)$
($\lambda \in \{ 1, -1\}$) and discover their relationships. Our approach
entails showing the existence of an infinite invertible matrix $N$ such that
\begin{equation}
\label{goal1}
N(P^TD)N^{-1}=(P_1^TD_1)\bigoplus (P_1^TD_1)\bigoplus \cdots \bigoplus (P_1^TD_1)\bigoplus \cdots
\end{equation}
and
\begin{equation}
\label{goal2}
D(N^{-1})^TD\,(PD)\,DN^TD=(P_1D_1)\bigoplus (P_1D_1)\bigoplus \cdots \bigoplus (P_1D_1)\bigoplus \cdots,
\end{equation}
which are infinite direct sums of copies of $P_1^TD_1$ and $P_1D_1$, respectively. This result will be
applied to characterize $\mathbf{E}_\lambda(PD)$ and $\mathbf{E}_\lambda(P^TD)$. Extending the work in
\cite{Choi}, we will also show that the columns of $P^{T\downarrow}$ and ${Q}^{T\downarrow}(0|0)$ form
bases for $\mathbf{E}_\lambda(P^TD)$. This will indeed allow us to investigate the relationships between
invariant sequences of the first and second kinds.

\section{The Eigenspaces of $P^TD$ and $PD$}
\par
\setcounter{num}{3}
\setcounter{equation}{0}
Let $A=[a_{ij}]$ $(i, j=0, 1, 2, \dots)$ be the matrix defined by
\[
a_{ij}=\left \{ \begin{array}{lr}
(-1)^{j-i}, &  {\rm if}~ i \leq j, \\
 0, &  {\rm if}~ i>j,
\end{array} \right.
\]
and let $J(a)$ denote the infinite Jordan block of the form
\[
    \left[ \begin{array}{ccccccc}
         ~~ a & 1   &          &         &    &  & \\
         ~~   & a   &  1       &         &   &  O& \\
              &     &  a       & 1       &    &  & \\
              &     &          &  \ddots &   \ddots &  & \\
              &    &          &         &  a & 1 &\\
              &     &   O       &         &   & a  &\ddots\\
              &     &          &         &    &    & \ddots~~
         \end{array} \right].
\]
It readily follows that $A^{-1}=J(1)$.

In the next two lemmas, we will construct an infinite matrix $N$ and its
inverse $M=N^{-1}$, which will give rise to similarity transformations of
$P^TD$ and $PD$ into direct sums as in (\ref{goal1}) and (\ref{goal2}).

\begin{lem}
\label{lemma1}
Let $m$ be a positive integer and let $N_{(m)}=H_{(m)} H_{(m-1)}\cdots H_{(1)}$,
where $H_{(k)}=I_{2k-2}\bigoplus F$ {\rm ($k=1, 2, \ldots, m$)}. Then
$$\lim_{m \to \infty}N_{(m)}=N=[n_{ij}^{\infty}]$$
is the infinite matrix defined by
$n_{00}^{\infty}=1$, $n_{0j}^{\infty}=0$ and $n_{i0}^{\infty}=0$ for $i, j=1, 2, \ldots$, and
\[
n_{ij}^{\infty}=\left \{ \begin{array}{lr}
(-1)^{j-i} {{\lfloor {{i-1}\over 2}\rfloor+j-i}\choose {\lfloor {{i-1}\over 2}\rfloor}},&  {\rm if}~~ 1\leq i \leq j, \\
 0, &  {\rm if}~~ i>j\geq 1.
\end{array} \right.
\]
\end{lem}
\begin{proof}
Let $m$ be a positive integer and let $N_{(m)}=H_{(m)}H_{(m-1)}\cdots H_{(1)}=[n_{ij}^{m}].$
We will prove that
\[
n_{ij}^{m}=\left \{ \begin{array}{lr}
(-1)^{j-i} {{k+j-i}\choose {k}},&  {\rm if}~~ i \leq j, \\
 0, &  {\rm if}~~ i>j
\end{array} \right.
\]
by induction on $m$, where
\[
k=\left \{ \begin{array}{lr}
\lfloor {{i-1}\over 2}\rfloor,&  {\rm if}~i=1, 2, \ldots, 2m;~~ j=1, 2, \ldots, \\
 m-1, &  {\rm if}~i=2m+1, 2m+2, \ldots;~j=1, 2, \ldots
\end{array} \right.
\]
The claim is clear for $m=1$. For $m=2$, by the construction of $H_{(1)}$ and $H_{(2)}$, we have
\[
n_{ij}^{2}=\left \{ \begin{array}{lr}
(-1)^{j-i} {{k+j-i}\choose {k}}, &  {\rm if}~~ i \leq j, \\
 0, &  {\rm if}~~ i>j,
\end{array} \right.
\]
where
\[
k=\left \{ \begin{array}{lr}
\lfloor {{i-1}\over 2}\rfloor, &  {\rm if}~i=1, 2, 3, 4;~~ j=1, 2, \ldots, \\
 1, &  {\rm if}~i=5, 6, \ldots;~j=1, 2, \ldots,
\end{array} \right.
\]
since for each $i=3, 4, \ldots$ and each $j=1, 2, \ldots$,
$$
n_{ij}^{2}=\sum_{l=i}^{\infty}(-1)^{l-i}(-1)^{j-l} {{j-l}\choose{0}} =(-1)^{j-i}\sum_{l=i}^{j} {{j-l}\choose{0}}
     =(-1)^{j-i} {{1+j-i}\choose{1}}.
$$
Let now $m \geq 3$. Then by the construction of $H_{(m)}$,
we have $n_{ij}^{m}=n_{ij}^{m-1}$ for all $i=1, 2, \ldots, 2m-2$
and all $j=1, 2, \ldots$
By the induction hypothesis,
\[
n_{ij}^{m}=\left \{ \begin{array}{lr}
(-1)^{j-i} {{k+j-i}\choose {k}},&  {\rm if}~~ i \leq j, \\
 0, &  {\rm if}~~ i>j,
\end{array} \right.
\]
where
\begin{eqnarray}
k=\left \{ \begin{array}{lr}
\lfloor {{i-1}\over 2}\rfloor, &  {\rm if}~i=1, 2, \ldots, 2m;~~ j=1, 2, \ldots, \\
 m-1, &  {\rm if}~i=2m+1, 2m+2, \ldots;~j=1, 2, \ldots,\\
\end{array} \right.
\end{eqnarray}
because
\begin{equation*}
\begin{split}
 n_{ij}^{m}&=\sum_{l=i}^{\infty}(-1)^{l-i}(-1)^{j-l} {{m-2+j-l} \choose {m-2}}\\
           &=(-1)^{j-i}\sum_{l=i}^{j}{{m-2+j-l} \choose {m-2}}=(-1)^{j-i} {{m-1+j-i} \choose {m-1}}
 \end{split}
\end{equation*}
for each $i=2m-1, 2m, \ldots$ and each $j=1,2, \ldots$
Thus, by $(2.1)$, we have that $\lim_{m \to \infty}N_{(m)}=N=[n_{ij}^{\infty}]$ given by
\[
n_{ij}^{\infty}=\left \{ \begin{array}{lr}
(-1)^{j-i} {{\lfloor {{i-1}\over 2}\rfloor+j-i}\choose {\lfloor {{i-1}\over 2}\rfloor}},&  {\rm if}~~ i \leq j, \\
 0, &  {\rm if}~~ i>j
\end{array} \right.
\]
for each $i, j=1, 2, \ldots$ Clearly, we have
$n_{00}^{\infty}=1$, $n_{0j}^{\infty}=0$ and $n_{i0}^{\infty}=0$ for $i, j=1, 2, \ldots$
by the construction of $H_{(l)}$ ($l=1,2, \ldots$),
and the proof is complete.
\end{proof}

The {\em difference sequence} $\Delta \mathbf{a}=[\Delta a_0, \Delta a_1 , \Delta a_2 , \ldots]^T$
of a sequence $\mathbf{a}=[a_0, a_1 , a_2 , \ldots ]^T$ is defined by $\Delta a_i=a_{i+1}-a_i$ for each $i=0, 1, 2, \ldots$
Let $\Delta^k \mathbf{a}=[\Delta^k a_0, \Delta^k a_1 , \Delta^k a_2 , \ldots ]^T$ $(k=0, 1, 2, \ldots)$ be the
$k$th difference sequence defined inductively by
$\Delta^k \mathbf{a}=\Delta(\Delta^{k-1} \mathbf{a})$, where $\Delta^0 \mathbf{a}=\mathbf{a}$.
The infinite matrix
\[ \left[
\begin{array}{cccc}
            a_0  &  a_1  &  a_2 & \cdots~\\
           \Delta a_0 &  \Delta a_1  & \Delta a_2 & \cdots~ \\
            \Delta^2 a_0 &  \Delta^2 a_1  & \Delta^2 a_2 & \cdots~ \\
              \vdots & \vdots  & \vdots & \ddots
            \end{array}\right]
\]
is called the {\em difference matrix} of $\mathbf{a}$. It is well known \cite{Brualdi}
that for each $n=0,1,2, \ldots$,
$$ a_n= a_0 {{n}\choose{0}} + \Delta a_0 {{n}\choose{1}} + \Delta^2 a_0 {{n}\choose{2}}+
  \cdots + \Delta^n a_0 {{n}\choose{n}},$$
which is used in the proof of the following lemma.

\begin{lem}
\label{lemma2}
Let $M=[m_{ij}^{\infty}]$ be the matrix with $m_{00}^{\infty}=1$, $m_{0j}^{\infty}=0$ and $m_{i0}^{\infty}=0$ for $i, j=1, 2, \ldots$, and
\[
m_{ij}^{\infty}=\left \{ \begin{array}{lr}
{{\lfloor {j\over 2}\rfloor}\choose{j-i}}, &  {\rm if}~~ i \leq j, \\
 0, &  {\rm if}~~ i>j
\end{array} \right. (i, j=1, 2, \ldots).
\]
Then $M=N^{-1}$, where $N$ is the limit matrix in {\em Lemma \ref{lemma1}}.
\end{lem}
\begin{proof}
Let $n_{ij}^{\infty}$ denote the $(i, j)$ entry of $N$ ($i, j=0, 1, 2, \ldots$). We would like to show that $$\sum_{l=0}^{\infty}n_{il}^{\infty}m_{lj}^{\infty}=\delta_{ij},$$ the Kronecker delta.
For $i=0$ or $j=0$, there is nothing to show, so let $i\geq 1$ and $j\geq 1$. If $i=j$ resp. $i>j$,
then clearly
$$\sum_{l=0}^{\infty}n_{il}^{\infty}m_{lj}^{\infty}=\sum_{l=i}^{i}n_{il}^{\infty}m_{li}^{\infty}={{\lfloor {i-1\over 2}\rfloor}\choose{\lfloor {i-1\over 2}\rfloor}}{{\lfloor {i\over 2}\rfloor}\choose{0}}=1
\;\;\mbox{resp.}\;\;
\sum_{l=0}^{\infty}n_{il}^{\infty}m_{lj}^{\infty}=0.$$
So it is enough to show that if $j=i+r$ with $r > 0$, then $\sum_{l=0}^{\infty}n_{il}^{\infty}m_{lj}^{\infty}=0$.
Let $k={\lfloor {i-1\over 2}\rfloor}$ and consider the sequence
$\mathbf{z}=\left[{{k+r}\choose{k}}, - {{k+r-1} \choose {k}}, {{k+r-2}\choose{k}}, \ldots, (-1)^{r-1} {{k+1}\choose{k}}, (-1)^r {{k}\choose{k}}, 0, \ldots\right]^T$.
We can construct the difference matrix $A=[a_{ij}]$ having $\mathbf{z}$
as its first column as follows:
\[
      A=\left[
       \begin{array}{cccc|ccc|c}
            ~ {k+r}\choose{k} & {k+r-1}\choose{k-1} & \cdots & {r}\choose{0} &   &  &    &\cdots~\\
            ~- {{k+r-1}\choose{k}}&-{{k+r-2}\choose{k-1}}&\cdots & -{{r-1}\choose{0}} &   &  & &   \cdots~\\
            ~{k+r-2}\choose{k} & {k+r-3}\choose{k-1} & \cdots & {r-2}\choose{0}&   &  &  &  \cdots~\\
            ~\vdots& \vdots& \cdots & \vdots&   & B &   & \cdots~\\
~(-1)^{r-2} {{k+2}\choose{k}} & (-1)^{r-2} {{k+1}\choose{k-1}} & \cdots & (-1)^{r-2} {{2}\choose{0}} &   &  &   &\cdots~\\
            ~(-1)^{r-1} {{k+1}\choose{k}} & (-1)^{r-1} {{k}\choose{k-1}} & \cdots & (-1)^{r-1} {{1}\choose{0}} &  &  &  & \cdots~\\
            ~(-1)^r {{k}\choose{k}} & (-1)^r {{k-1}\choose{k-1}} & \cdots & (-1)^r {{0}\choose{0}}&  &  &   & \cdots~\\\cline{1-8}
                        ~0 & 0 & 0 & 0 & 0  &\cdots & 0 & \cdots~\\
             ~\vdots & \vdots & \vdots & \vdots &\vdots & \vdots &  \vdots & \ddots~
                   \end{array}
\right],
\]
where $B$ is the $(r+1)\times (r+1)$ matrix given by
\[
      B=\left[
       \begin{array}{c|c|c|c|c|c|c|c|c}
            ~    0   &  0    & 0     & 0     &  \cdots     &    0  &   0   &  0 &  *  ~\\\cline{8-8}
            ~    0   &  0    & 0     & 0     & \cdots      &    0  &    0  & *  &  *   ~\\\cline{7-7}
            ~ \vdots&\vdots&\vdots&\vdots& \iddots         &    0  &    *  &   * &  *   ~\\\cline{6-6}
            ~  0    &  0   &   0  &  0   & \iddots &   *   &    *   & *   & *    ~\\\cline{5-5}
~              0    & 0    &   0  &  0   &\iddots &   \vdots    &  \vdots     &  \vdots  & \vdots    ~\\\cline{4-4}
            ~  0    &  0   &   0  &  *   & \cdots       &   *    &   *    &  *  & *    ~\\\cline{3-3}
            ~  0    &  0   &  *   &  *   & \cdots       &   *    &   *    &  *  & *    ~\\\cline{2-2}
              ~0    &  *   &  *   &   *  &  \cdots      &   *    &   *    &  *  & *     ~\\\cline{1-1}
               *    &  *   &  *   &  *   &     \cdots   &    *   &   *    &  *  & * ~
                   \end{array}
\right].
\]
Thus $a_{0j}=0$ for $j=k+1, \ldots, k+r$. Since $k+1 \leq \lfloor {i+r\over 2}\rfloor \leq \lfloor {i-1+r+r\over 2}\rfloor =k+r$, we obtain
\begin{equation*}
\begin{split}
\sum_{l=0}^{\infty}n_{il}^{\infty}m_{lj}^{\infty}&=\sum_{l=i}^{j}n_{il}^{\infty}m_{lj}^{\infty}=\sum_{l=i}^{i+r}(-1)^{l-i} {{k+l-i}\choose{k}} {{t}\choose{i+r-l}}\\
&=\sum_{s=0}^{t}(-1)^{r+s} {{k+r-s}\choose{k}} {{t}\choose{s}}=(-1)^r a_{0t}=0,
\end{split}
\end{equation*}
where $k=\lfloor {i-1\over 2}\rfloor$ and  $t=\lfloor {i+r\over 2}\rfloor$, completing the proof.
\end{proof}
Let $M_{(t)}=U_{(1)} U_{(2)}\cdots U_{(t)}$, where $U_{(j)}=I_{2j-2}\bigoplus J(1)$ ($j=1, 2, \ldots, t$).
Then, by Lemmas \ref{lemma1} and \ref{lemma2}, we have $\lim_{t \to \infty}M_{(t)}=M$, which is the matrix in Lemma \ref{lemma2}.
In fact we have
\[
      N= \left[ \begin{array}{ccccccccc}
           ~ 1 & 0 & 0 & 0 & 0  & 0 & 0 & 0 &   \cdots~\\
            ~0 & 1 &-1 & 1 & -1  & 1 & -1 & 1 &   \cdots~\\
            ~0 & 0 & 1 & -1 & 1  & -1 & 1 & -1 &   \cdots~\\
            ~0 & 0 & 0 & 1 & -2  & 3 & -4 & 5 &   \cdots~\\
            ~0 & 0 & 0 & 0 & 1  & -2 & 3 & -4 &   \cdots~\\
            ~0 & 0 & 0 & 0 & 0  & 1 & -3 & 6 &   \cdots~\\
            ~0 & 0 & 0 & 0 & 0  & 0 & 1 & -3 &   \cdots~\\
            ~0 & 0 & 0 & 0 & 0  &0 & 0 & 1 &   \cdots~\\
             ~\vdots & \vdots & \vdots & \vdots &\vdots & \vdots & \vdots & \vdots & \ddots~
                   \end{array} \right]
\]
and
\[
      M=N^{-1} = \left[ \begin{array}{ccccccccc}
           ~ 1 & 0     &     0 &    0 & 0    & 0      & 0 &  0    &  \cdots~\\
            ~0 & 1     &     1 &    0 & 0    & 0      & 0 &  0    &  \cdots~\\
            ~0 & 0     &     1 &    1 & 1    & 0      & 0 &  0    &  \cdots~\\
            ~0 & 0     &     0 &    1 & 2    & 1      & 1 &  0    &   \cdots~\\
            ~0 & 0     &     0 &    0 & 1    & 2      & 3 &  1    & \cdots~\\
            ~0 & 0     &     0 &    0 & 0    & 1      & 3 &  3    & \cdots~\\
            ~0 & 0     &     0 &    0 & 0    & 0      & 1 &  3    &  \cdots~\\
            ~0 & 0     &    0  &  0   & 0    & 0      & 0 &  1    & \cdots~\\
         \vdots& \vdots&\vdots &\vdots&\vdots& \vdots &\vdots & \vdots& \ddots~

             \end{array} \right].
\]

We can now state and prove the similarity transformations of $PD$ and $P^TD$ claimed
in (\ref{goal1}) and (\ref{goal2}).
\begin{thm}
\label{similarity}
Let $P=\left[    \begin{matrix}
           {i}\choose{j}
         \end{matrix} \right] (i, j=0, 1, \ldots)$ and $D={\rm diag}(1, -1, 1, -1, \ldots)$. Then,
\begin{itemize}
  \item [\rm (a)] $N(P^TD)M=(P_1^TD_1)\bigoplus (P_1^TD_1)\bigoplus \cdots \bigoplus (P_1^TD_1)\bigoplus \cdots$,
  \item [\rm (b)] $(DM^TD)\, (PD)\, (DN^TD)=(P_1D_1)\bigoplus (P_1D_1)\bigoplus \cdots \bigoplus (P_1D_1)\bigoplus \cdots$,
\end{itemize}
where
$P_1$ resp. $D_1$ are the leading $2\times 2$ submatrices of $P$ resp. $D$,
and $N$ resp. $M$ are the matrices in {\rm Lemmas} $3.1$ resp. $3.2$.
\end{thm}
\begin{proof}
(a) For each $j=1, 2, \ldots$, let $H_{(j)}=I_{2j-2}\bigoplus F$ and $U_{(j)}=I_{2j-2}\bigoplus J(1)$.
Let $m$ be an arbitrary positive integer with $n=2m+1$. First, we will show by
induction on $m$ that
\begin{equation*}
\left(H_{(m)}H_{(m-1)}\cdots H_{(1)}\right)P^TD \left(U_{(1)}U_{(2)}\cdots U_{(m)}\right)=Z \bigoplus (P^TD),
\end{equation*}
where $Z$ is an $n \times n$ matrix such that
$Z=\overbrace{(P^T_1D_1)\bigoplus (P^T_1D_1)\bigoplus \cdots \bigoplus (P^T_1D_1)}^{m}$.
When $m=1$, the first row of $H_{(1)}P^TDU_{(1)}$ is clearly
$[1, -1, 0, 0, \ldots]$.
Since for each $i$ and $j$ with $i, k=1, 2, \ldots$,
\[
(H_{(1)}P^TD)_{ik}=(-1)^k\biggl({{k}\choose{i}}-{{k}\choose{i+1}}+\cdots+(-1)^{k}{{k}\choose{k}}\biggr)
=(-1)^k{{k-1}\choose{i-1}},
\]
the second row of $H_{(1)}P^TDU_{(1)}$ is $[0, -1, 1, -1, 1, \ldots]U_{(1)}=[0, -1, 0, 0, 0, \ldots]$.
For each $i$ and $j$ with $i, j \geq 2$, we have
\[
(H_{(1)}P^TDU_{(1)})_{ij}=(-1)^{j-1}{{j-2}\choose{i-1}}+(-1)^j{{j-1}\choose{i-1}}=(-1)^{j-2}{{j-2}\choose{i-2}},
\]
and so $H_{(1)}P^TDU_{(1)}=(P^T_1D_1)\bigoplus (P^TD)$.
By induction on $m$, it follows
\begin{equation*}
N_{(m)}P^TDU_{(m)}=\overbrace{(P^T_1D_1)\bigoplus (P^T_1D_1)\bigoplus \cdots \bigoplus (P^T_1D_1)}^{m}\bigoplus (P^TD),
\end{equation*}
where $N_{(m)}=H_{(m)}H_{(m-1)}\cdots H_{(1)}$ and $M_{(m)}=U_{(1)}U_{(2)}\cdots U_{(m)}$. Since $m$ was an arbitrary positive integer,
\[
NP^TDM=({P^T_1}D_1)\bigoplus ({P^T_1}D_1)\bigoplus \cdots \bigoplus ({P^T_1}D_1) \bigoplus \cdots,
\]
where $N=\lim_{m \to \infty}N_{(m)}$ and  $M=\lim_{m \to \infty}M_{(m)}$.\\
(b) It follows directly from (a) that
$$(DM^TD)PD(DN^TD)=(P_1D_1)\bigoplus (P_1D_1)\bigoplus \cdots \bigoplus (P_1D_1)\bigoplus \cdots,$$
completing the proof.
\end{proof}
Let $\mathbf e_i$ denote the $i$th column of the identity matrix $I$ ($i=0, 1, \ldots$).
Then $\mathbf{E}_{\lambda} (P^TD)$ and $\mathbf{E}_{\lambda} (PD)$ ($\lambda \in \{ 1, -1\}$)
can be characterized via Theorem \ref{similarity} as follows:

\begin{thm}
\label{bases}
Let $N$ resp. $M=N^{-1}$ be the matrices in {\rm Lemma \ref{lemma1} resp.
Lemma \ref{lemma2}}. Then the following hold:
\begin{itemize}
  \item [\rm (a)] $\{M \mathbf e_0, M\mathbf e_2, M\mathbf e_4, \ldots\}$ is a basis for $\mathbf{E}_1(P^TD)$.
  \item [\rm (b)] $\{M(\mathbf e_0+2 \mathbf e_1), M(\mathbf e_2+2\mathbf e_3), \ldots\}$
              is a basis for $\mathbf{E}_{-1}(P^TD)$.
  \item [\rm (c)]  $\{DN^TD(2\mathbf e_0+\mathbf e_1), DN^TD(2\mathbf e_2+\mathbf e_3) \ldots\}$
               is a basis for $\mathbf{E}_1(PD)$.
  \item [\rm (d)] $\{DN^TD \mathbf e_1, DN^TD\mathbf e_3, DN^TD\mathbf e_5, \ldots\}$
              is a basis for $\mathbf{E}_{-1}(PD)$.
\end{itemize}
\end{thm}
\begin{proof}
(a) Let $\mathbf B_{(1)}=\{\mathbf e_0, \mathbf e_2, \mathbf e_4, \ldots\}$
and $\mathbf x=[x_0, x_1, x_2, \ldots]^T \in \mathbf{E}_1(NP^TDM)$. Then,
by Theorem \ref{similarity}, we have
$$(NP^TDM-I)\mathbf x=\biggl( \bigoplus_{i=1}^{\infty} \left[ \begin{array}{cc}
          0 & -1   \\
          0 & -2
            \end{array} \right] \biggr) \mathbf x=\mathbf 0,$$
which implies that $x_{i}=t_i, x_{i+1}=0$ for each $i=0, 2, 4, \ldots$ and $t_i \in \mathbb R$.
So $\mathbf B_{(1)}$ spans $\mathbf{E}_1(NP^TDM)$ and  since
$\mathbf e_0, \mathbf e_2, \mathbf e_4, \ldots$ are linearly independent,
$\mathbf B_{(1)}$ is a basis for $\mathbf{E}_1(NP^TDM)$.
Therefore $\{M \mathbf e_0, M\mathbf e_2, M\mathbf e_4, \ldots\}$ is a basis for $\mathbf{E}_1(P^TD)$.\\
(b) Let $\mathbf B_{(-1)}=\{\mathbf e_0+2 \mathbf e_1, \mathbf e_2+ 2\mathbf e_3, \ldots\}$
and $\mathbf y=[y_0, y_1, y_2, \ldots]^T \in \mathbf{E}_{-1}(NP^TDM)$. Then,
by Theorem \ref{similarity}, we have
\[
 (NP^TDM+I)\mathbf y=\biggl( \bigoplus_{i=1}^{\infty} \left[ \begin{array}{cc}
         2 & -1   \\
         0 & 0
            \end{array}\right] \biggr) \mathbf y=\mathbf 0,
\]
which implies that $\left[ \begin{array}{c}
          y_{i}  \\
          y_{i+1}
            \end{array}\right] =s_i \left[ \begin{array}{c}
          1   \\
          2
            \end{array}\right]$
for each $i=0, 2, 4, \ldots$ and $s_i \in \mathbb R$.
So $\mathbf B_{(-1)}$ spans $\mathbf{E}_1(NP^TDM)$ and since
$\mathbf e_0+2 \mathbf e_1, \mathbf e_2+ 2\mathbf e_3, \ldots$ are
linearly independent, $\mathbf B_{(-1)}$ is a basis for $\mathbf{E}_{-1}(NP^TDM)$.
Therefore $\{M(\mathbf e_0+2 \mathbf e_1), M(\mathbf e_2+2\mathbf e_3), \ldots\}$
is a basis for $\mathbf{E}_{-1}(P^TD)$. \\
Clauses (c) and (d) can be proven similarly.
\end{proof}

Consider now
\[
      {P^T\hspace{0.01cm}^\downarrow}=\left[\begin{array}{ccccccccc}
           ~ 1 & 0 & 0 & 0 & 0  & 0 & 0 & 0 &   \cdots~\\
            ~0 & 1 & 0 & 0 & 0  & 0 & 0 & 0 &   \cdots~\\
            ~0 & 1 & 1 & 0 & 0  & 0 & 0 & 0 &   \cdots~\\
            ~0 & 0 & 2 & 1 & 0  & 0 & 0 & 0 &   \cdots~\\
            ~0 & 0 & 1 & 3 & 1  & 0 & 0 & 0 &   \cdots~\\
            ~0 & 0 & 0 & 3 & 4  & 1 & 0 & 0 &   \cdots~\\
            ~0 & 0 & 0 & 1 & 6  & 5 & 1 & 0 &   \cdots~\\
            ~0 & 0 & 0 & 0 & 4  &10 & 6 & 1 &   \cdots~\\
             ~\vdots & \vdots & \vdots & \vdots &\vdots & \vdots & \vdots & \vdots & \ddots~
                   \end{array}\right].
\]
The following is a matrix expression of Theorem \ref{bases}; the last two
clauses appeared in \cite{Choi}.
\begin{cor}
\label{columns}
Let $Q=P+\left[\begin{array}{c|c}
1 & {\mathbf 0}^T\\\cline{1-2}
{\mathbf 0}& P
\end{array} \right]$
and $D={\rm diag}((-1)^0, (-1)^1, \ldots)$, where $P=\left[\begin{matrix}
           {i}\choose{j}
         \end{matrix} \right] (i, j=0, 1, 2, \ldots)$.
Then the following hold:
\begin{itemize}
  \item [\rm (a)] The columns of $P^{T\downarrow}$ form a basis for $\mathbf{E}_1(P^TD)$.
  \item [\rm (b)] The columns of $Q^{T\downarrow}(0|0)$ form a basis for $\mathbf{E}_{-1}(P^TD)$.
  \item [\rm (c)] The columns of
$\left[\begin{array}{c}
{\mathbf 0}^T\\
{P\hspace{0.01cm}^\downarrow}
\end{array} \right]$
form a basis for $\mathbf{E}_{-1}(PD)$.
  \item [\rm (d)] The columns of
${Q\hspace{0.01cm}^\downarrow}$ form a basis for $\mathbf{E}_1(PD)$.
\end{itemize}
\end{cor}
\begin{proof}
(a) Let $(i, j)$ be a pair of integers with $i, j\geq 0$.
The $i$th component of $M{\mathbf e}_{2j}$ equals
${\displaystyle j \choose {2j-i}}$ when $i \leq 2j$, and equals
$0$, otherwise; thus the $j$th column of $P^{T\downarrow}$ is
$$M{\mathbf e}_{2j}=(\overbrace{0, \ldots, 0}^j,{j \choose 0}, {j \choose 1}, \ldots, {j \choose j}, 0, 0, \ldots)^T.$$

\noindent
(b) The $i$th component of
$M({\mathbf e}_{2j}+2{\mathbf e}_{2j+1})$ equals
$$
{j \choose {2j-i}}+2{j \choose {2j-i+1}}={{j+1} \choose {2j-i+1}}+{j \choose {2j-i+1}},
$$
when $i\leq 2j+1$, and equals $0$ otherwise; thus the $j$th column of
$$ P^{T\downarrow}(0|0)+\left[\begin{array}{c|c}
1 & {\mathbf 0}^T\\\cline{1-2}
{\mathbf 0}& P^T
\end{array} \right]^{T\downarrow}(0|0)=Q^{T\downarrow}(0|0) $$ is
\begin{equation*}
\begin{split}
M({\mathbf e}_{2j}+2{\mathbf e}_{2j+1})&=[\overbrace{0, \ldots, 0}^j,{j+1 \choose 0}, {j+1 \choose 1}, \ldots, {j+1 \choose j+1}, 0, 0, \ldots]^T\\
&+[\overbrace{0, \ldots, 0}^{j+1},{j \choose 0}, {j \choose 1}, \ldots, {j \choose j}, 0, 0, \ldots]^T.
\end{split}
\end{equation*}
Using the fact that for each $i$ and $j$ with $i \geq j\geq 1$,
$$(DN^TD)_{ij}=(-1)^i(-1)^{i-j} {{\lfloor {{j-1}\over 2}\rfloor+i-j}\choose
{\lfloor {{j-1}\over 2}\rfloor}}(-1)^j,$$
clauses (c) and (d) can be proven similarly.
\end{proof}


\section{Invariant sequences of two kinds: Relations and examples}
\par
\setcounter{num}{4}
\setcounter{equation}{0}
We begin by some basic examples of (inverse) invariant sequences.

\begin{ex}
\label{fibomat}
{\rm
It follows from Corollary \ref{columns} (c) and (d) that
the Fibonacci sequence $\mathbf F$ is an invariant sequence of the first kind
and the Lucas sequence $\mathbf L$ is an inverse invariant sequence of the first kind.

The $i$th row of $P^{T\downarrow}$ ($i=0,1,\ldots$) is
\begin{equation}
\label{fibo}
[\overbrace{0, \ldots, 0}^{\lceil {{i}\over 2}\rceil}, {{\lceil {{i}\over 2}\rceil}\choose {\lfloor {{i}\over 2}\rfloor}},\ldots, {{i-2} \choose {2}}, {{i-1} \choose {1}}, {{i} \choose {0}}, 0, 0, \ldots],
\end{equation}
from which we can get that $J(0)\mathbf F$ is an invariant sequence of the
second kind and $J(0)\mathbf L$ is an inverse invariant sequence of the second kind.
Recall that $J(0)$ is the infinite Jordan block with $0$ in the main diagonal.
}
\end{ex}

In the following theorem, we provide a general mechanism for transforming
invariant into inverse invariant sequences, and vice versa.
\begin{thm}
\label{transform}
Let $J(\lambda)$ denote the infinite Jordan block with $\lambda$ in the main diagonal.
Then the following hold:
\begin{itemize}
  \item [\rm (a)] If $\mathbf x$ is an invariant sequence of the second kind, then $(J(1)+J(0))^T{\mathbf x}$
is an inverse invariant sequence of the second kind.
  \item [\rm (b)] If $\mathbf x$ is an inverse invariant sequence of the second kind,
  then $J(2)^{-1}J(0) {\mathbf x}$ is an invariant sequence of the second kind.
  \item [\rm (c)] If $\mathbf x$ is an invariant sequence of the first kind, then
  $\left(-J(0)J(-2)^{-1}\right)^T {\mathbf x}$ is an inverse invariant sequence of the first kind.
  \item [\rm (d)] If $\mathbf x$ is an inverse invariant sequence of the first kind,
  then $(J(-1)+J(0)){\mathbf x}$ is an invariant sequence of the first kind.
  \end{itemize}
\end{thm}
\begin{proof}
(a) If $\mathbf x$ is an invariant sequence of the second kind, then by
Corollary \ref{columns} (a), there exists $\mathbf b \in \mathbb{R}^\infty$
such that $P^{T\downarrow}\mathbf b =\mathbf x$. Let $L=[l_{ij}]$ ($i, j=0, 1, 2,\ldots$)
be the infinite lower triangular matrix defined by
$$ l_{ij}=\left \{ \begin{array}{lr}
(-1)^{i+j}, &  {\rm if}~ i \geq j, \\
 0, &  {\rm if}~ i<j.
\end{array} \right.$$
Then $LP^{T\downarrow}(0|0)=P^{T\downarrow}$
because for each $i$ and $j$ with $i, j \geq 0$, the $(i, j)$ entry of $LP^{T\downarrow}(0|0)$ equals
$$ (-1)^{i+j}{j+1 \choose 0}+(-1)^{i+j+1}{j+1 \choose 1}+\cdots+(-1)^{i+j+i-j}{j+1 \choose {i-j}}={j \choose i-j}$$
when $i \geq j$, and equals $0$ otherwise, which coincides with the $(i, j)$ entry of $P^{T\downarrow}$.
Thus,
\begin{equation}
\label{2.2}
LQ^{T\downarrow}(0|0)=P^{T\downarrow}+L\left[\begin{array}{c}
{\mathbf 0}^T\\
P^{T\downarrow}
\end{array} \right].
\end{equation}
From the fact that $L^{-1}=J(1)^T$ and by Corollary \ref{columns} (b), it follows that
\begin{equation*}
Q^{T\downarrow}(0|0)\mathbf b=J(1)^TP^{T\downarrow}\mathbf b+\left[\begin{array}{c}
{\mathbf 0}^T\\
P^{T\downarrow}
\end{array} \right]\mathbf b=J(1)^T{\mathbf x}+[0, {\mathbf x}^T]^T=(J(1)+J(0))^T{\mathbf x}
\end{equation*}
is an inverse invariant sequence of the second kind.\\
(b) Let $\mathbf x$ be an inverse invariant sequence of the second kind.
From (\ref{2.2}) and $J(0)J(1)^T=J(1)$, it readily follows that
$$ J(2)^{-1}J(0)Q^{T\downarrow}(0|0)=P^{T\downarrow}$$
and so by Corollary \ref{columns} (b), we have that
$J(2)^{-1}J(0){\mathbf x}$ is an invariant sequence of the second kind. \\
(c) Since $Q^{\downarrow}=P^{\downarrow}+\left[\begin{array}{c|c}
1 & {\mathbf 0}^T\\\cline{1-2}
0& {\mathbf 0}^T\\\cline{1-2}
 {\mathbf 0}& P^{\downarrow}
\end{array} \right]$, we have
\begin{equation*}
J(0)^TQ^{\downarrow}=J(0)^TP^{\downarrow}+\left[\begin{array}{c|c}
0& {\mathbf 0}^T\\\cline{1-2}
1 & {\mathbf 0}^T\\\cline{1-2}
0& {\mathbf 0}^T\\\cline{1-2}
 {\mathbf 0}& P^{\downarrow}
\end{array} \right].
\end{equation*}
So $\Omega \left(
\left[\begin{array}{c}
{\mathbf 0}^T\\
Q^{\downarrow}
\end{array} \right]-
\left[\begin{array}{c}
{\mathbf 0}^T\\
P^{\downarrow}
\end{array} \right]\right)=\left[\begin{array}{c}
{\mathbf 0}^T\\
P^{\downarrow}
\end{array} \right]$, because
${i\choose i}+{i\choose i+1}+\cdots+ {i\choose i+k}={{i+1} \choose {i+k+1}}$ for $i, k=0, 1,\ldots$,
where $\Omega$ is the infinite $(0,1)$-matrix with $1$'s everywhere on and below its main diagonal.
Since $-J(-1)^T=\Omega^{-1}$, we get
$(I-J(-1)^T)^{-1}\left[\begin{array}{c}
{\mathbf 0}^T\\
Q^{\downarrow}
\end{array} \right]=\left[\begin{array}{c}
{\mathbf 0}^T\\
P^{\downarrow}
\end{array} \right]$, which implies that $\left(-J(0)J(-2)^{-1}\right)^T {\mathbf x}$ is an inverse
invariant sequence of the first kind. \\
Clauses (d) easily follows from (c) similarly.
\end{proof}

Let $\tau_1={{1+\sqrt{5}}\over 2}$ and $\tau_2={{1-\sqrt{5}}\over 2}$.
It is well known that $F_n={1 \over \sqrt{5}}\tau_1^{n}-{1 \over \sqrt{5}}\tau_2^{n}$
and $L_n=\tau_1^{n}+\tau_2^{n}$ where $F_n$ and $L_n$ are the $n$th terms
of ${\mathbf F}$ and ${\mathbf L}$, respectively ($n=0, 1, 2, \ldots$)  \cite{Brualdi}.
Since $J(0){\mathbf F}$ is an invariant sequence of the second kind by (\ref{fibo}),
it follows from Theorem \ref{transform} (a) that
$$J(1)^TJ(0){\mathbf F}+J(0)^TJ(0){\mathbf F}=J(1)^TJ(0){\mathbf F}+{\mathbf F}$$
is an inverse invariant sequence of the second kind. In fact, the $n$th term ($n=0, 1, 2, \ldots$)
of  $J(1)^TJ(0){\mathbf F}+{\mathbf F}$ is
$$ {1 \over \sqrt{5}}\tau_1^{n+2}-{1 \over \sqrt{5}}\tau_2^{n+2}+
                {1 \over \sqrt{5}}\tau_1^{n}-{1 \over \sqrt{5}}\tau_2^{n}=L_{n+1}, $$
which implies that $J(1)^TJ(0){\mathbf F}+{\mathbf F}=J(0){\mathbf L}$.
On the other hand, since for each $i$ and $j$ with $i, j=0, 1, 2,\ldots$,
$$J(2)^{-1}=\left \{ \begin{array}{lr}
(-1)^{j-i}({1 \over 2})^{j-i+1}, &  {\rm if}~ i \leq j, \\
 0, &  {\rm if}~ i>j,
\end{array} \right.$$
and the $n$-th term of  $J(2)^{-1}J(0)^2{\mathbf L}$ is
$$
\sum_{k=0}^{\infty}(-1)^k(1/2)^{k+1}(\tau_1^{k+2+n}+\tau_2^{k+2+n})={1 \over \sqrt{5}}\tau_1^{n+1}-{1 \over \sqrt{5}}\tau_2^{n+1}=F_{n+1},$$
we have $J(2)^{-1}J(0)^2{\mathbf L}=J(0){\mathbf F}$, which is an invariant sequence of the second
kind by Theorem \ref{transform} (b), as stated in Example \ref{fibomat}.

It directly follows from Theorem \ref{transform} (d) that $(J(0)+J(-1)){\mathbf F}$,
namely $\,\mathbf L$, is an invariant sequence of the first kind.
Since for each $i$ and $j$ with $i, j=0, 1, 2,\ldots$,
$$-J(-2)^{-1}=\left \{ \begin{array}{lr}
({1 \over 2})^{j-i+1}, &  {\rm if}~ i \leq j, \\
 0, &  {\rm if}~ i>j,
\end{array} \right.$$
we obtain $\left(-J(-2)^{-1}\right)^T [0,{\mathbf L}^T]^T=\mathbf F$, which
is an inverse invariant sequence of the first kind by Theorem \ref{transform} (c),
because for $n=0, 1, 2, \ldots$,
$$\sum_{k=0}^{n-1}(1/2)^{n-k}L_k=(1/2)^{n}~\sum_{k=0}^{n-1}((2\tau_1)^{k}+(2\tau_2)^{k})
={1 \over \sqrt{5}}\tau^{n}-{1 \over \sqrt{5}}\tau_2^{n}=F_n.$$

The sequence $\mathbf B=(B_0, B_1, \ldots)^T$ defined by
$B_0=1$ and $\sum_{k=0}^{n} {{n+1} \choose k}B_k=0~(n \geq 1)$ comprises the Bernoulli numbers and
$D \mathbf B$ is an invariant sequence of the first kind  \cite{Sun},
which also follows directly from the fact that $PDD \mathbf B=D \mathbf B$.
A new inverse invariant sequence of the first kind from the Bernoulli numbers $\mathbf B$ is
provided next. See Table $1$ for explicit members of these sequences.
\begin{cor}
	Let $\mathbf B=(B_0, B_1, \ldots)^T$ be the Bernoulli numbers.
	Then the sequence $\mathbf K=(K_0, K_1, K_2, \ldots)^T$ defined by
	\begin{equation*}
	K_0=0,~ K_n=\sum_{k=0}^{n-1}(1/2)^{n-k}(-1)^kB_k \;~(n=1, 2, \ldots)
	\end{equation*}
	is an inverse invariant sequence of the first kind.
\end{cor}
\begin{proof}
	It follows from Theorem \ref{transform} (c) that
	$$ (-J(-2)^{-1})^TJ(0)^TD \mathbf B=(-J(-2)^{-1})^T[0,D{\mathbf B}^T]^T $$
	is an inverse invariant sequence of the first kind.
	Notice now that the sequence  ${\mathbf K}$ in the statement is indeed equal to
	$(-J(-2)^{-1})^T[0,D{\mathbf B}^T]^T. $
\end{proof}


\begin{table}[!ht] 
\caption[]{\footnotesize (Inverse) invariant sequences of the first kind associated with the Bernoulli numbers.} \label{mytable}
\begin{tabular}{cccccccccccccccccc}
\hline
n&0&1&2&3&4&5&6&7&8&9&10&11&12&$\cdots$\\
\hline
$\mathbf B$&1&$-{1\over2}$&${1\over6}$&0&$-{1\over 30}$&0&${1\over 42}$&0&$-{1\over 30}$&0&${5\over 66}$&0&$-{691 \over 2730}$&$\cdots$\\
\hline
$D\mathbf B$&1&${1\over2}$&${1\over6}$&0&$-{1\over 30}$&0&${1\over 42}$&0&$-{1\over 30}$&0&${5\over 66}$&0&$-{691 \over 2730}$&$\cdots$\\
\hline
$\mathbf K$&0&${1 \over 2}$&${1\over2}$&${1 \over 3}$&${1 \over 6}$&${1 \over 15}$&${1 \over 30}$&${1 \over 35}$&${1 \over 70}$&$-{1 \over 105}$&$-{1 \over 210}$&${41 \over 1155}$&${41 \over 2310}$&$\cdots$\\
\hline
\end{tabular}
\end{table}


By Theorem \ref{transform} (d), we get $D \mathbf B=(J(0)+J(-1)){\mathbf K}$, since the first component of $(J(0)+J(-1)){\mathbf K}$
is clearly $(-1)^0B_0$, and for $i=1, 2, \ldots$, $i$th component of $(J(0)+J(-1)){\mathbf K}$ is
\begin{equation*}
-\sum_{k=0}^{i-1}(1/2)^{i-k}(-1)^kB_k+2\sum_{k=0}^{i}(1/2)^{i+1-k}(-1)^kB_k=(-1)^iB_i.
\end{equation*}
By Corollary \ref{columns} and Theorem \ref{transform}, we can directly get more (inverse) invariant sequences of the
first and second kind as follows:
\begin{cor}
	\label{newcor1}
	For a positive integer $n$, let
	$\mathbf \Phi_n={\mathcal S}_{l_n}{\mathcal S}_{l_{n-1}} \ldots {\mathcal S}_{l_3}S_{l_2}{\mathcal S}_{l_1}$,
	where
	\begin{equation*}
	{\mathcal S}_{l_i}=\begin{cases}
	P^{T\downarrow},&\text{if $l_i=1$},\\
	J(1)^T+J(0)^T, & \text{if $l_i=-1$},
	\end{cases}
	\end{equation*}
	and let $\tilde{\mathbf \Phi}_n=\tilde{\mathcal S}_{l_n}\tilde{\mathcal S}_{l_{n-1}} \ldots \tilde{\mathcal S}_{l_3}\tilde{\mathcal S}_{l_2}\tilde{\mathcal S}_{l_1}$ where
	\begin{equation*}
	\tilde{\mathcal S}_{l_i}=\begin{cases}
	J(2)^{-1}J(0),&\text{if $l_i=1$},\\
	Q^{T\downarrow}(0|0), & \text{if $l_i=-1$}
	\end{cases}
	\end{equation*}
	for $i=1, 2, \ldots, n$. Then for $\mathbf x \in \mathbb{R}^\infty$, we have the following:
	\begin{itemize}
		\item [\rm (a)] If $n$ is odd (even) and $l_i=(-1)^{i+1}$ for $i=1, 2, \ldots, n$, then $\mathbf \Phi_n{\mathbf x}$
		is an invariant (inverse invariant) sequence of the second kind.
		\item [\rm (b)] If $n$ is odd (even) and $l_i=(-1)^i$ for $i=1, 2, \ldots, n$, then $\tilde{\mathbf \Phi}_n{\mathbf x}$
		is an inverse invariant (invariant) sequence of the second kind.
	\end{itemize}
\end{cor}

\begin{ex}
{\rm
	It follows from (\ref{fibo}) and Corollary \ref{newcor1} (a) that
	for $\mathbf x=[x_0, x_1, x_2, \ldots]^T \in \mathbb{R}^\infty$, $\mathbf \Phi_2 \mathbf x=\mathbf y$
	is an inverse invariant sequence of the second kind, where $\mathbf y=[y_0, y_1, y_2, \ldots]^T$ with $$y_i=\sum_{t=\lfloor {i \over 2}\rfloor}^{i+1}\left({\binom{t}{i-1-t}}+{\binom{t}{i+1-t}}\right)x_t \;\;\; (i=0, 1, \ldots). $$
	For example, let $\mathbf x=[0, 0, 0, 0, 0, 0, 0, 1, 0, 0, \ldots]^T$.
	This results to the inverse invariant sequence of the second kind $\mathbf y=[y_0, y_1, y_2, \ldots]^T$,
	where
	$$ y_i={\binom{7}{i-1-7}}+{\binom{7}{i+1-7}} \;\;\; (i=0, 1, 2, \ldots). $$
	That is, $y_0=\cdots=y_5=0$, $y_{16}=y_{17}=\cdots=0$, and by direct calculation, one can compute the nonzero components $y_6, y_7, \ldots, y_{15}$;
	e.g., $y_{10}={\binom{7}{10-1-7}}+{\binom{7}{10+1-7}}=56$.
		
}	
\end{ex}

\begin{cor}
\label{newcor2}
	For a positive integer $n$, let $\mathbf \Psi_n={\mathcal T}_{l_n}{\mathcal T}_{l_{n-1}} \ldots {\mathcal T}_{l_3}{\mathcal T}_{l_2}{\mathcal T}_{l_1}$,
	where
	\begin{equation*}
	{\mathcal T}_{l_i}=\begin{cases}
	Q^{\downarrow},&\text{if $l_i=1$},\\
	\left(-J(-2)^{-1}\right)^TJ(0)^T, & \text{if $l_i=-1$}
	\end{cases}
	\end{equation*}
	and let $\tilde{\mathbf \Psi}_n=\tilde{\mathcal T}_{l_n}\tilde{\mathcal T}_{l_{n-1}} \ldots \tilde{\mathcal T}_{l_3}\tilde{\mathcal T}_{l_2}\tilde{\mathcal T}_{l_1}$, where
	\begin{equation*}
	\tilde{\mathcal T}_{l_i}=\begin{cases}
	J(0)+J(-1),&\text{if $l_i=1$},\\
	\left[\begin{array}{c}
	{\mathbf 0}^T\\
	{P\hspace{0.01cm}^\downarrow}
	\end{array} \right], & \text{if $l_i=-1$}
	\end{cases}
	\end{equation*}
	for $i=1, 2, \ldots, n$. Then for $\mathbf x \in \mathbb{R}^\infty$, we have the following:
	\begin{itemize}
		\item [\rm (a)] If $n$ is odd (even) and $l_i=(-1)^{i+1}$ for $i=1, 2, \ldots, n$,
		then $\mathbf \Psi_n{\mathbf x}$ is an invariant (inverse invariant) sequence of the first kind.
		\item [\rm (b)] If $n$ is odd (even) and $l_i=(-1)^i$ for $i=1, 2, \ldots, n$,
		then $\tilde{\mathbf \Psi}_n{\mathbf x}$ is an inverse invariant (invariant) sequence of the first kind.
	\end{itemize}
\end{cor}


\begin{ex}
	{\rm
		
		For $i=1, 2, \ldots$, the $i$th row of $\left[\begin{array}{c}
		{\mathbf 0}^T\\
		{P\hspace{0.01cm}^\downarrow}
		\end{array} \right]$ is	
		$[{\binom{i-1}{0}},{\binom{i-2}{1}}, \ldots, {{\lceil {{i-1}\over 2}\rceil}\choose
			{\lfloor {{i-1}\over 2}\rfloor}}, 0, 0, \ldots].$
		It follows that
		for $\mathbf x=[x_0, x_1, x_2, \ldots]^T \in \mathbb{R}^\infty$, $\tilde{\mathbf \Psi}_2 \mathbf x=\mathbf y$
		is an invariant sequence of the first kind, where  by Corollary \ref{newcor2} (b), $\mathbf y=(y_0, y_1, y_2, \ldots)^T$ satisfies
		$$y_i=\sum_{t=0}^{\lfloor {i \over 2}\rfloor}\left({\binom{i-1-t}{t-1}}+{\binom{i-t}{t}}\right)x_t.$$
		For example, let $\mathbf x=[0, 0, 0, 0, 0, 0, 0, 1, 0, 0, \ldots]^T$. This results to the invariant sequence of
		the first kind $\mathbf y=[y_0, y_1, y_2, \ldots]^T$, where
		$$y_i={\binom{i-1-7}{6}}+{\binom{i-7}{7}} \;\;\;(i=1, 2, \ldots).$$
		Thus $y_0=y_1=\cdots=y_{13}=0$ and one can find by direct calculation, the nonzero components $y_{14}, y_{15}, \ldots$;  e.g., $y_{14}={\binom{6}{6}}+{\binom{7}{7}}=2$.
	}	
\end{ex}

In the next two theorems, we obtain direct relationships among (inverse) invariant sequences
of the first kind and (inverse) invariant sequences of the second kind.

\begin{thm}
\label{new2}
Let $\mathbf{x}$ and $\mathbf{y}$ be, respectively, either
\\
{\small
{\rm (i)} an invariant sequence of the first kind and an inverse invariant
sequence of the second kind,
\\
\hspace*{.25in} or
\\
{\rm (ii)} an inverse invariant sequence of the first kind and an invariant
sequence of the second kind.}\\
Then $\mathbf{x}^TD\mathbf{y}=0$.
\end{thm}
\begin{proof}
From $PD\mathbf{x}=\lambda\mathbf{x}$ and $P^TD\mathbf{y}=-\lambda\mathbf{y}$, we get $\mathbf{x}^TDP^TD\mathbf{y}=\lambda\mathbf{x}^TD\mathbf{y}$, which
implies that $2\lambda\mathbf{x}^TD\mathbf{y}=0$ for $\lambda \in \{ 1, -1\}$.
\end{proof}
The following lemma is a useful tool for proving the final theorem.
\begin{lem}
\label{new3}
For every positive integer $n$,
\begin{itemize}
  \item [\rm (a)] the columns of $(P+D)^n$ are invariant sequences of the first kind, and the columns of $(P-D)^n$ are
              inverse invariant sequences of the first kind.
  \item [\rm (b)] the columns of $(P^T+D)^n$ are invariant sequences of the second kind, and the columns of $(P^T-D)^n$ are
  inverse invariant sequences of the second kind.
  \end{itemize}
\end{lem}
\begin{proof}
Let $n$ be a positive integer. Then $$PD(P+D)^n=PD(P+D)(P+D)^{n-1}=(P+D)^n$$ and $$P^TD(P^T-D)^n=P^TD(P^T-D)(P^T-D)^{n-1}=-(P^T-D)^n,$$
which respectively imply that each column of $(P+D)^n$ is an invariant sequence of the first kind,
and each column of $(P^T-D)^n$ is an inverse invariant sequence of the second kind. The other assertions of the theorem
follow similarly.
\end{proof}
\vskip 0.5mm
The following theorem is a form of converse of Theorem \ref{new2}.
\begin{thm}
\label{new4}
Let ${\mathbf x_i}\in \mathbb{R}^\infty~(i=0, 1, 2, \ldots)$,  ${\mathbf y} \in \mathbb{R}^\infty\setminus\{{\mathbf 0}\}$,
and let $X=[\mathbf x_0, \mathbf x_1, \mathbf x_2, \ldots]$. Then for each $i=0, 1, 2, \ldots$,
\begin{itemize}
  \item [\rm (a)] if ${\mathbf x_i}^TD{\mathbf y}=0$ and $X=P+D~(X=P-D)$, then ${\mathbf y}$ is an inverse invariant (invariant)
  sequence of the second kind, and ${\mathbf x_i}$ is an invariant (inverse invariant) sequence of the first kind.
  \item [\rm (b)] if ${\mathbf x_i}^TD{\mathbf y}=0$ and $X=P^T+D~(X=P^T-D)$, then ${\mathbf y}$ is an inverse invariant (invariant) sequence
  of the first kind and ${\mathbf x_i}$ is an inverse invariant (invariant) sequence of the second kind.
  \end{itemize}
\end{thm}
\begin{proof}
If ${\mathbf x_i}^TD{\mathbf y}=0$ and $X=P+D$, then since $X^TD \mathbf{y}=(P^T+D)D\mathbf{y}=\mathbf 0$, we get $P^TD\mathbf{y}=-\mathbf{y}$ and by Lemma \ref{new3}, $PD \mathbf{x}_i=\mathbf{x}_i$. So ${\mathbf y}$ is an inverse invariant sequence of the second kind and ${\mathbf x_i}$ is an invariant sequence of the first kind.
This proves the first case of part (a). For the case of $X=P-D$ and part (b), the results can be shown similarly.
\end{proof}

\vskip 0.5cm
\section*{References}
\bibliography{mybibfile}

\end{document}